\documentclass[12pt,reqno]{amsart}
\usepackage{amsmath, amsfonts, amssymb, amsthm}
\def\S{\mathfrak{S}}
\def\T{\mathcal{T}}
\def\hH{\mathcal{H}}
\def\L{\mathcal{L}}

\def\P{\mathcal{P}}
\theoremstyle{plain}
\newtheorem{theorem}{Theorem}

\newtheorem{corollary}{Corollary}
\theoremstyle{definition}
\newtheorem*{Ack}{Acknowledgment}
\theoremstyle{remark}

\begin{document}
\title{A note on the Polignac numbers}
\author{Hao Pan}
\email{haopan1979@gmail}
\address{Department of Mathematics, Nanjing University,
Nanjing 210093, People's Republic of China}
\begin{abstract}
Suppose that $k\geq 3.5\times 10^6$ and $\hH=\{h_1,\ldots,h_{k_0}\}$ is admissible. Then for any $m\geq 1$, the set
$$
\{m(h_j-h_i):\,h_i<h_j\}
$$
contains at least one Polignac number.
\end{abstract}
 \maketitle

A recent huge breakthrough on prime number theory is Zhang's brilliant work, which asserts that
$$
\liminf_{n\to\infty}(p_{n+1}-p_n)\leq 7\times10^7,
$$
where $p_n$ denotes the $n$-th prime. For a set $\hH=\{h_1,h_2,\ldots,h_{k_0}\}$ of positive integers, we say $\hH$ is \textit{admissible} if
$$\nu_p(\hH)<p$$ for every prime $p$, where $\nu_p(\hH)$ denotes the number of distinct residue classes occupied by those $h_i$ modulo $p$. Zhang proved that if $k_0\geq 3.5\times 10^6$ and $\hH=\{h_1,\ldots,h_{k_0}\}$ is admissible, then
for sufficiently large $x$, there exists $n\in[x,2x]$ such that
$$
\{n+h_1,n+h_2,\cdots,n+h_{k_0}\}
$$
contains at least two primes.

In fact, we may give the following ``cheap'' extension for Zhang's theorem.
\begin{theorem}\label{t1}
Let $k_0\geq 3.5\times 10^6$ and $A>0$. Suppose that $x$ is sufficiently large and $1\leq q\leq(\log x)^A$. If $\hH=\{h_1,\ldots,h_{k_0}\}$ is admissible and $(q,h_1\cdots h_{k_0})=1$, there exists $n\in[x,2x]$ such that
$$
\{qn+h_1,qn+h_2,\cdots,qn+h_{k_0}\}
$$
contains at least two primes.
\end{theorem}
The proof of Theorem \ref{t1} is just a copy of the original one of Zhang's. The only modification is to set
$$
P(n)=\prod_{i=1}^{k_0}(qn+h_i)
\qquad\text{and}
\qquad
\S=\prod_{\substack{p\text{ prime}\\ p\nmid q}}\bigg(1-\frac{\nu_p(\hH)}{p}\bigg)\cdot \prod_{p\text{ prime}}\bigg(1-\frac{1}{p}\bigg)^{-k_0}.
$$
Then the difference between
\begin{align*}\T_2:=\sum_{\substack{d_0\mid\P\\ (d_0,q)=1}}\sum_{\substack{d_1\mid\P\\ (d_1,q)=1}}\sum_{\substack{d_2\mid\P\\ (d_2,q)=1}}\frac{\mu(d_1d_2)\varrho_2(d_0d_1d_2)}{\phi(d_0d_1d_2)}g(d_0d_1)g(d_0d_2)
\end{align*}
and
\begin{align*}
\T_2^*:=&\sum_{\substack{(d_0,q)=1}}\sum_{\substack{(d_1,q)=1}}\sum_{\substack{(d_2,q)=1}}\frac{\mu(d_1d_2)\varrho_2(d_0d_1d_2)}{\phi(d_0d_1d_2)}g(d_0d_1)g(d_0d_2)\\
=&\frac{\phi(q)}{q}\bigg(\frac{1}{(k_0+2l_0+1)!}\binom{2l_0+2}{l_0+1}\S(\log D)^{k_0+2l_0+1}+o(\L^{k_0+2l_0+1})\bigg)
\end{align*}
can be bounded by
$$
\frac{\phi(q)}{q}\bigg(\frac{\kappa_2}{(k_0+2l_0+1)!}\binom{2l_0+2}{l_0+1}\S(\log D)^{k_0+2l_0+1}+o(\L^{k_0+2l_0+1})\bigg).
$$
Then applying \cite[Theorem 2]{Zhang}, we can obtain the expected lower bound for
\begin{align*}
S_2:=\sum_{j=1}^{k_0}\bigg(\sum_{n\sim x}\theta(qn+h_j)\bigg)\lambda(n)^2
=\frac{q}{\phi(q)}\cdot k_0\T_2^*x+O(x\L^{-B}).
\end{align*}

As an immediate consequence of Theorem \ref{t1}, we have
\begin{corollary}\label{c1}
Suppose that $0\leq b<q$ and $(b,q)=1$. Let $p_{n}^{(b,q)}$ denote the $n$-th prime of the form $qm+b$. Then
$$
\liminf_{n\to\infty}\frac{p_{n+1}^{(b,q)}-p_{n}^{(b,q)}}{q}\leq 7\times 10^7.
$$
\end{corollary}
Apparently, Corollary \ref{c1} follows from the evident fact that the admissibility of $\{b+qh_1,b+qh_2,\cdots,b+qh_{k_0}\}$ is implied by the one of 
$\{h_1,h_2,\cdots,h_{k_0}\}$.

Another application of Theorem \ref{t1} is on the Polignac numbers \cite{Polignac}. A positive even number $d$ is called a Polignac number, if there exist infinitely many $n$ such that
$$
p_{n+1}-p_n=d.
$$
Of course, it is believed that every positive even number is a Polignac number. 
Recently, combining Zhang's techniques with some lemmas from \cite{PM}, Pintz \cite{Pintz} proved that the set of all Polignac numbers has a positive lower density.
Now, we shall prove that
\begin{theorem}\label{t2}
Suppose that $k_0\geq 3.5\times 10^6$ and  $\hH=\{h_1,\ldots,h_{k_0}\}$ is admissible. Let
$$
\sigma(\hH)=\{h_j-h_i:\,h_i<h_j\}.
$$
Then for any $m\geq 1$, the set
$$
m\cdot\sigma(\hH)=\{md:\,d\in\sigma(\hH)\}
$$
contains at least one Polignac number.
\end{theorem}
\begin{proof}
Without loss of generality, assume that $h_1<h_2<\cdots<h_{k_0}$. Let
$$
X=\{a\in[mh_1,mh_{k_0}]:\,a\equiv mh_1\pmod{2},\ a\not\in\{mh_1,\ldots,mh_{k_0}\}\}.
$$
Assume that $X=\{a_1,a_2,\ldots,a_l\}$. Arbitrarily choose distinct primes $p_1,p_2,\ldots,p_l>mh_{k_0}$. 
Let $b>0$ be an integer such that 
$
b\equiv 1\pmod{m}
$
and
$
b\equiv-a_j\pmod{p_j}$
for $1\leq j\leq l$. Let $q=mp_1p_2\cdots p_l$. Since $(b,m)=1$, $\{b+mh_1,\cdots,b+mh_{k_0}\}$ is admissible. And for each $j$, noting that 
$p_j\mid b+a_j$ and $p_j>mh_{k_0}$, we must have 
$$
\prod_{i=1}^{k_0}(b+mh_i)\not\equiv0\pmod{p_j}.
$$
That is, $q$ is prime to $(b+mh_1)\cdots(b+mh_{k_0})$. By Theorem \ref{t1}, there exist infinitely many $n$ such that
$$
\{qn+b+mh_1,qn+b+mh_2,\ldots,qn+b+mh_{k_0}\}
$$
contains at least two primes. Let $n_1,n_2,n_3,\ldots$ be all such $n$. For each $s$, we may choose a pair $(i_s,j_s)$ with $i_s<j_s$ such that both $qn_s+b+mh_{i_s}$ and $qn_s+b+mh_{j_s}$ are prime, but $qn_s+b+mh_{k}$ are composite for all $i_s<k<j_s$. Clearly there exists a pair $(i_*,j_*)$ such that
$$
|\{s:\,(i_s,j_s)=(i_*,j_*)\}|
$$
is infinity. That is, $qn+b+mh_{i_*}$ and $qn+b+mh_{j_*}$ are prime for infinitely many $n$. But according to the definition of $q$, for any $a_j\in(mh_{i_*},mh_{j_*})$, $qn+b+a_j$ is divisible by $p_j$. So $qn+b+mh_{i_*}$ and $qn+b+mh_{j_*}$ must be two consecutive primes, i.e., $m(h_{j_*}-h_{i_*})$ is a Polignac number. We are done.
\end{proof}

\begin{Ack} I thank Professor Zhi-Wei Sun for his helpful discussions on Zhang's theorem.
\end{Ack}

\end{document}